\documentclass[a4paper,12pt]{article}

\usepackage[utf8]{inputenc}
\usepackage{amsfonts}
\usepackage{amssymb}
\usepackage{amsthm}
\usepackage{amsmath}

\usepackage{enumerate} 
\usepackage{enumitem} 

\usepackage[a4paper,twoside,top=2.5cm, bottom=2cm, left=2.3cm, right=2.3cm]{geometry} 

\usepackage{graphicx}
\usepackage{dsfont}
\usepackage{authblk}

\usepackage{color}
\definecolor{hanblue}{rgb}{0.27, 0.42, 0.81}
\definecolor{red}{rgb}{1.0, 0.0, 0.0}
\usepackage[colorlinks, citecolor=blue,linkcolor=blue, urlcolor = blue]{hyperref}

\usepackage[english,capitalize]{cleveref}

\allowdisplaybreaks

\newcommand*{\uthe}{\mathbf{u}(\theta)}

\newcommand*{\R}{\mathbb{R}}

\newcommand*{\de}{\,\ensuremath{\mathrm d}}

\theoremstyle{plain}
\newtheorem{thm}{Theorem}[section]
\newtheorem{lem}[thm]{Lemma}
\newtheorem{prop}[thm]{Proposition}

\theoremstyle{definition}
\newtheorem{defn}[thm]{Definition}

\theoremstyle{remark}
\newtheorem{rem}[thm]{Remark}

\numberwithin{equation}{section}

\begin{document}

\title{Polynomial-order oscillations in geometric discrepancy}
\author{Thomas Beretti}
\affil{International School for Advanced Studies (SISSA)}
\date{\today}
\maketitle

\begin{abstract}
\noindent
Let $C\subset\mathbb{R}^2$ be a convex body, and for a positive integer $N$, let $\mathcal{P}$ be a configuration of $N$ points in $[0,1)^2$. The discrepancy of $\mathcal{P}$ with respect to $C$ is defined by
    \begin{equation*}
    \mathcal{D}(\mathcal{P},\, C)=\sum_{\mathbf{p}\in\mathcal{P}}\sum_{\mathbf{n}\in\mathbb{Z}^2}\mathds{1}_C(\mathbf{p}+\mathbf{n})-N|C|,
\end{equation*}
and one may estimate how $\mathcal{P}$ deviates from uniformity by averaging the latter quantity over a family of sets. When considering quadratic averages over translated and dilated copies of $C$, one gets the \textit{homothetic quadratic discrepancy}
\begin{equation*}
    \mathcal{D}_2(\mathcal{P},\, C)=\int_{0}^{1}\int_{[0,1)^2}\left|\mathcal{D}( \mathcal{P},\,\boldsymbol{\tau}+\delta C)\right|^2\,{\rm d}\boldsymbol{\tau}\,{\rm d} \delta.
\end{equation*}
We investigate the behaviour of the optimal \textit{homothetic quadratic discrepancy}, that is
\begin{equation*}
    \inf_{\# \mathcal{P}=N} \mathcal{D}_2(\mathcal{P},\, C)\quad\text{as}\quad N\to+\infty.
\end{equation*}
Beck~\cite{MR915529} and Beck and Chen~\cite{MR1489133} showed that the optimal \textit{h.q.d.} of convex polygons has an order of growth of $\log N$, and more recently, Brandolini and Travaglini~\cite{MR4358540} proved that the optimal \textit{h.q.d.} of planar convex bodies with a $\mathcal{C}^2$ boundary has an order of growth of $N^{1/2}$. We show that, in general, a single order of growth for the optimal \textit{h.q.d.} need not exist. First, by an implicit geometric construction of $C$, we obtain prescribed oscillations between $\log N$ and $N^{1/2}$. Second, by a subtler design of $\partial C$ and via Fourier-analytic methods, we obtain prescribed polynomial-order oscillations in the range $N^\alpha$ with $\alpha\in(2/5,1/2)$. Moreover, we show that the set of planar convex bodies whose optimal \textit{h.q.d.} does not admit a single order of growth is residual in the (Hausdorff) metric space of planar convex bodies.
\end{abstract}

\tableofcontents

\section{Introduction}

We define the two-dimensional torus as $\mathbb{T}^2=\mathbb{R}^2/\mathbb{Z}^2$,
and we identify $[0,1)^2$ with $\mathbb{T}^2$ via the quotient map $\boldsymbol{\tau}\mapsto \boldsymbol{\tau}+\mathbb{Z}^2$. Let $C$ be a planar convex body (that is, a compact convex set with non-empty interior), and let $\mathds{1}_{C}$ stand for its characteristic (indicator) function. Consider the periodization functional ${\mathfrak{P}}\colon L^1(\mathbb{R}^2)\to L^1(\mathbb{T}^2)$ defined by
\begin{equation*}
    {\mathfrak{P}}\{\mathds{1}_C\}(\mathbf{x})=\sum_{\mathbf{n}\in\mathbb{Z}^2}\mathds{1}_C(\mathbf{x}+\mathbf{n}).
\end{equation*}

The theory of irregularities of distribution, often referred to as discrepancy theory, quantifies how point distributions in a space deviate from uniformity (see \cite{MR2683232} and \cite{MR4391422} for a general overview and applications). In geometric discrepancy theory, historically initiated by Roth's seminal work~\cite{MR66435}, one measures such a difference using sets as test functions and compares the number of points that fall within them with the expected value. We introduce the following notions of discrepancy.
\begin{defn}
Let $C\subset\mathbb{R}^2$ be a convex body, and let $\mathcal{P}_N$ be a configuration of $N$ points (that is, a collection of $N$ not necessarily distinct points counted with multiplicity) in $\mathbb{T}^2$. We define the \emph{discrepancy} of $\mathcal{P}_N$ with respect to $C$ as 
    \begin{equation}\label{Discrepancy1}
    \mathcal{D}(\mathcal{P}_N,\, C)=\sum_{\mathbf{p}\in\mathcal{P}_N}{\mathfrak{P}}\{\mathds{1}_C\}(\mathbf{p})-N|C|.
\end{equation}
Let $\boldsymbol{\tau}\in[0,1)^2$ be a translation parameter and let $\delta\geq0$ be a dilation parameter. We define the \emph{homothetic quadratic discrepancy} (in short, \textit{h.q.d.}) of $\mathcal{P}_N$ with respect to $C$ as
\begin{equation}\label{Discrepancy2}
    \mathcal{D}_2(\mathcal{P}_N,\, C)=\int_{0}^{1}\int_{\mathbb{T}^2}\left|\mathcal{D}( \mathcal{P}_N,\,\boldsymbol{\tau}+\delta C)\right|^2\de\boldsymbol{\tau}\de \delta,
\end{equation}
where $\de\boldsymbol{\tau}$ is the normalised Haar measure on $\mathbb T^2$ and $\de\delta$ is the Lebesgue measure on $[0,1]$.
\end{defn}

It is a classical question in the theory of irregularities of distribution to determine the optimal behaviour of the \textit{h.q.d.} as $N$ (the number of samples per unit square) goes to infinity. In other words, we are concerned with the quantity
\begin{equation*}
    \inf_{\# \mathcal{P}=N} \mathcal{D}_2(\mathcal{P},\, C)\quad\text{as}\quad N\to+\infty.
\end{equation*}
Surprisingly, the qualitative behaviour of this quantity depends strongly on the geometry of $C$: for instance, when $C$ is the unit square one obtains a logarithmic order of growth, while when $C$ is the unit disk one obtains a polynomial order. We now present sharp estimates for more general bodies, but first, we introduce suitable notation for orders of growth.

Let $f$ and $g$ be two functions defined on the same domain, and let $q$ be an object; we write $f\lesssim_q g$ to indicate that there exists a positive value $c$, that may depend on $q$, such that $f\leq cg$ in the whole domain. If we omit $q$ in the subscript, then $c$ is a positive constant. If the $\lesssim_q$ (or $\lesssim$) holds in both senses, then we write $f\approx_q g$ (or $f\approx g$).

First, we state a result on the optimal \textit{h.q.d.} with respect to polygons. Throughout the paper, we assume that $N$ is an integer greater than $1$.
\begin{thm}[Beck-Chen]\label{BC}
    Let $C$ be a convex polygon. Then, it holds
    \begin{equation*}
        \inf_{\# \mathcal{P}=N} \mathcal{D}_2(\mathcal{P},\, C)\approx_C \log N.
    \end{equation*}
\end{thm}

In particular, the lower bound is due to a classical work of Beck~\cite{MR915529}, while the upper bound follows from a subsequent work of Beck and Chen~\cite{MR1489133}.

In contrast with the latter result, one has the following result on the optimal \textit{h.q.d.} for convex bodies with sufficiently regular boundaries.

\begin{thm}[Brandolini-Travaglini]\label{BT1}
    Let $C$ be a planar convex body. If the boundary of $C$ is $\mathcal{C}^2$ (regardless of curvature), then  it holds
    \begin{equation*}
        \inf_{\# \mathcal{P}=N} \mathcal{D}_2(\mathcal{P},\, C)\approx_C N^{1/2}.
    \end{equation*}
\end{thm}
The upper bound is well-known and can be easily shown via a probabilistic argument employing a randomly shifted lattice (for example, see \cite{MR3897016}). On the other hand, the lower bound is due to a recent work of Brandolini and Travaglini~\cite{MR4358540} and requires subtle Fourier-analytic techniques; these provide a starting point for our work. In the same paper, the following result is shown, notably recovering a polynomial order of growth across an interval of exponents.

\begin{thm}[Brandolini-Travaglini]\label{BT2}
    Let $\alpha\in[2/5,1/2)$. Then, there exists a planar convex body $C_\alpha$ such that
    \begin{equation*}
        \inf_{\# \mathcal{P}=N} \mathcal{D}_2(\mathcal{P},\, C_\alpha)\approx_\alpha N^\alpha.
    \end{equation*}
\end{thm}

In the current paper, we aim to show that, in general, the optimal \textit{h.q.d.} does not need to exhibit a single order of growth, and prescribed stationary orders of growth can be achieved. For this purpose, we present two general methods for constructing special sets whose optimal \textit{h.q.d.} continuously switch between different orders of growth (logarithmic or polynomial).

Our first method employs a simple geometric argument. This allows prescribed oscillations from a logarithmic order of $\log N$ to a polynomial order of $N^{1/2}$, and vice versa. We state our result as follows.
\begin{thm}\label{main1}
    Let $\left\{\alpha_i\right\}_{i\in\mathbb{N}}$ be in $\{0,1/2\}$, let $\left\{\varepsilon_i\right\}_{i\in\mathbb{N}}$ be in $(0,\varepsilon]$ for some small positive $\varepsilon$, and let $\left\{q_i\right\}_{i\in\mathbb{N}}$ be positive integers. Then there exists a planar convex body $C$ and an increasing sequence $\left\{ N_i\right\}_{i\in\mathbb{N}}\subset\mathbb{N}$ such that, for every $i\in\mathbb{N}$ such that $\alpha_i=0$, it holds
    \begin{equation*}
        \log^{1-\varepsilon_i}N\leq\inf_{\# \mathcal{P}=N}\mathcal{D}_2(\mathcal{P},\, C)\leq \log^{1+\varepsilon_i}N\quad\text{for every}\quad N\in[N_i,N_i+q_i],
    \end{equation*}
    while, for every $i\in\mathbb{N}$ such that $\alpha_i=1/2$, it holds
    \begin{equation*}
        N^{1/2-\varepsilon_i}\leq\inf_{\# \mathcal{P}=N}\mathcal{D}_2(\mathcal{P},\, C)\leq N^{1/2+\varepsilon_i}\quad\text{for every}\quad N\in[N_i,N_i+q_i].
    \end{equation*}
\end{thm}
This first method obtains the planar convex body $C$ as the limit of an increasing sequence of planar convex bodies; thus, $C$ is constructed implicitly. The simplicity of our first method comes at the cost of limited knowledge on the geometry of the boundary of $C$. In particular, we rely on Lemma~\ref{UseLem}, showing that two bodies with a small symmetric difference have close discrepancies. Further, by pairing the just-mentioned lemma with a classical topological argument, we show that the set of planar convex bodies whose optimal \textit{h.q.d.} has a single order of growth is meagre in the space of planar convex bodies endowed with the Hausdorff metric.

On the other hand, our second method provides a direct construction. This employs harmonic analysis techniques and subtle geometric estimates on the Fourier transform of sets. We obtain prescribed polynomial-order oscillations in the range $N^\alpha$ with $\alpha\in(2/5,1/2)$, and state our second result as follows.
\begin{thm}\label{main2}
    Let $\left\{\alpha_i\right\}_{i\in\mathbb{N}}$ be in $(2/5,1/2)$, let $\left\{\varepsilon_i\right\}_{i\in\mathbb{N}}$ be in $(0,\varepsilon]$ for some small positive $\varepsilon$, and let $\left\{q_i\right\}_{i\in\mathbb{N}}$ be positive integers. Then there exists a planar convex body $C$ and an increasing sequence $\left\{ N_i\right\}_{i\in\mathbb{N}}\subset\mathbb{N}$ such that, for every $i\in\mathbb{N}$, it holds
    \begin{equation*}
        N^{\alpha_i-\varepsilon_i}\leq\inf_{\# \mathcal{P}=N}\mathcal{D}_2(\mathcal{P},\, C)\leq N^{\alpha_i+\varepsilon_i}\quad\text{for every}\quad N\in[N_i,N_i+q_i].
    \end{equation*}
\end{thm}
The second method involves constructing the boundary of $C$ “by hand”; as we will see, it suffices to design the boundary only in a small neighbourhood of a point, while leaving the remaining part as $\mathcal{C}^2$. Moreover, we note that we can alternate exponents within the same range as in Theorem~\ref{BT2}.

To the best of our knowledge, these two results settle the first examples of sets for which the optimal \textit{h.q.d.} has prescribed stationary orders of growth. Interestingly, if one further considers full rotations in the averaging process (that is, averaging over the whole similarity group), it is a classical result, obtained independently by Beck~\cite{MR0906524} and Montgomery~\cite[Ch.~6]{MR1297543},  that one always obtains an order of growth of $N^{1/2}$. The reader may consult a recent work by Gennaioli and the author \cite{beretti2024fouriertransformbvfunctions} for further details on the interplay between discrepancy and geometric measure theory.

In the pages that follow, Section~\ref{S1} contains the implicit geometric construction that leads to Theorem~\ref{main1}, while Section~\ref{S2} contains the direct construction that proves Theorem~\ref{main2}. In particular, the proof of Theorem~\ref{main2} requires different techniques to establish lower and upper bounds, and therefore, these are treated in distinct subsections.

    \section{The first method}\label{S1}
    Recalling the hypotheses of Theorem~\ref{main1}, we have that $\left\{\alpha_i\right\}_{i\in\mathbb{N}}$ is a sequence of exponents with values in $\{0,1/2\}$,  $\left\{\varepsilon_i\right\}_{i\in\mathbb{N}}$ is a (possibly decreasing) sequence in $(0,\varepsilon]$ for some small positive $\varepsilon$, and $\left\{q_i\right\}_{i\in\mathbb{N}}$ is a (possibly increasing) sequence of positive integers. We remark that $\alpha_i$ prescribes the target order of growth, while $\varepsilon_i$ and $q_i$ quantify, respectively, how closely and for how long we approximate it.
    
    We begin by presenting a simple but essential auxiliary result.

    \begin{lem}\label{UseLem}
        Let $\mathcal{P}$ be a configuration of $N$ points in $\mathbb{T}^2$, and let $A$ and $B$ be two planar sets such that $\rm{diam}(A)\leq1$, $A\supset B$, and $|A\setminus B|\leq\eta\leq1$. Then, it holds
        \begin{equation*}
            \left|\mathcal{D}_2(\mathcal{P},A)-\mathcal{D}_2(\mathcal{P},B)\right|\leq 4N^2\eta^{1/2}.
        \end{equation*}
    \end{lem}
    \begin{proof}
        By the reverse triangle inequality, it holds
        \begin{equation}\label{U1}
            \left|\mathcal{D}_2(\mathcal{P},A)-\mathcal{D}_2(\mathcal{P},B)\right|^2\leq 4N^2\int_0^1\int_{\mathbb{T}^2}\left|\mathcal{D}(\mathcal{P},\boldsymbol{\tau}+\delta A)-\mathcal{D}(\mathcal{P},\boldsymbol{\tau}+\delta B)\right|^2\de\boldsymbol{\tau}\de\delta.
        \end{equation}
        Now, since $A\supset B$, it holds
        \begin{equation}\label{U2}
            \mathcal{D}(\mathcal{P},\boldsymbol{\tau}+\delta A)-\mathcal{D}(\mathcal{P},\boldsymbol{\tau}+\delta B)=\mathcal{D}\left(\mathcal{P},\boldsymbol{\tau}+\delta (A\setminus B)\right).
        \end{equation}
        Since $\rm{diam}(A)\leq 1$, then, for any $\mathbf{p}\in\mathbb{T}^2$, the set of $\boldsymbol{\tau}$'s such that
        \begin{equation*}
            \mathds{1}_{\boldsymbol{\tau}+\delta(A\setminus B)}(\mathbf{p})\neq 0
        \end{equation*}
        has measure $\delta^2|A\setminus B|\leq\delta^2\eta$. Therefore, it holds
        \begin{equation*}
            \int_{\mathbb{T}^2}\left|\mathds{1}_{\boldsymbol{\tau}+\delta(A\setminus B)}(\mathbf{p})\right|^2\de\boldsymbol{\tau}\leq\delta^2\eta.
        \end{equation*}
        Now, by applying the Cauchy–Schwarz inequality, we obtain
        \begin{equation*}
        \begin{split}
        \int_0^1\int_{\mathbb{T}^2}\left|\mathcal{D}\left(\mathcal{P},\,\boldsymbol{\tau}+\delta (A\setminus B)\right)\right|^2\de\boldsymbol{\tau}\de\delta&\leq2\int_0^1\int_{\mathbb{T}^2}\left|\sum_{\mathbf{p}\in\mathcal{P}}\mathds{1}_{\boldsymbol{\tau}+\delta(A\setminus B)}(\mathbf{p})\right|^2\de\boldsymbol{\tau}\de\delta+2N^2\eta^2\\
        &\leq2N\sum_{\mathbf{p}\in\mathcal{P}}\int_0^1\int_{\mathbb{T}^2}\left|\mathds{1}_{\boldsymbol{\tau}+\delta(A\setminus B)}(\mathbf{p})\right|^2\de\boldsymbol{\tau}\de\delta+2N^2\eta^2\\&\leq2N^2\eta+2N^2\eta^2\leq4N^2\eta.
        \end{split}
        \end{equation*}
        Combining the latter with \eqref{U1} and \eqref{U2} yields the claim.
    \end{proof}
    
    The latter result is all we need to proceed with the proof of Theorem~\ref{main1}. We start by considering a sequence of convex planar bodies $\left\{C_i\right\}_{i\in\mathbb{N}}$ to be chosen later, but such that each of their diameters is less than $1$ and
    \begin{equation*}
        \partial C_i\quad\text{is}\quad\begin{cases}
            \text{a polygon}&\text{if}\quad \alpha_i=0\\
            \text{a }\mathcal{C}^2\text{ curve} &\text{if}\quad \alpha_i=1/2\\
        \end{cases}.
    \end{equation*}
    Now, let $\left\{\eta_i\right\}_{i\in\mathbb{N}}$ be a (decreasing) sequence with values in $(0,1]$, but to be chosen later as well. Moreover, for every $i\in\mathbb{N}$, we further impose that $C_i\subset C_{i+1}$ and that $|C_{i+1}\setminus C_i|\leq \eta_{i+1}$.\begin{figure}
        \centering
        \includegraphics[width=0.9\linewidth]{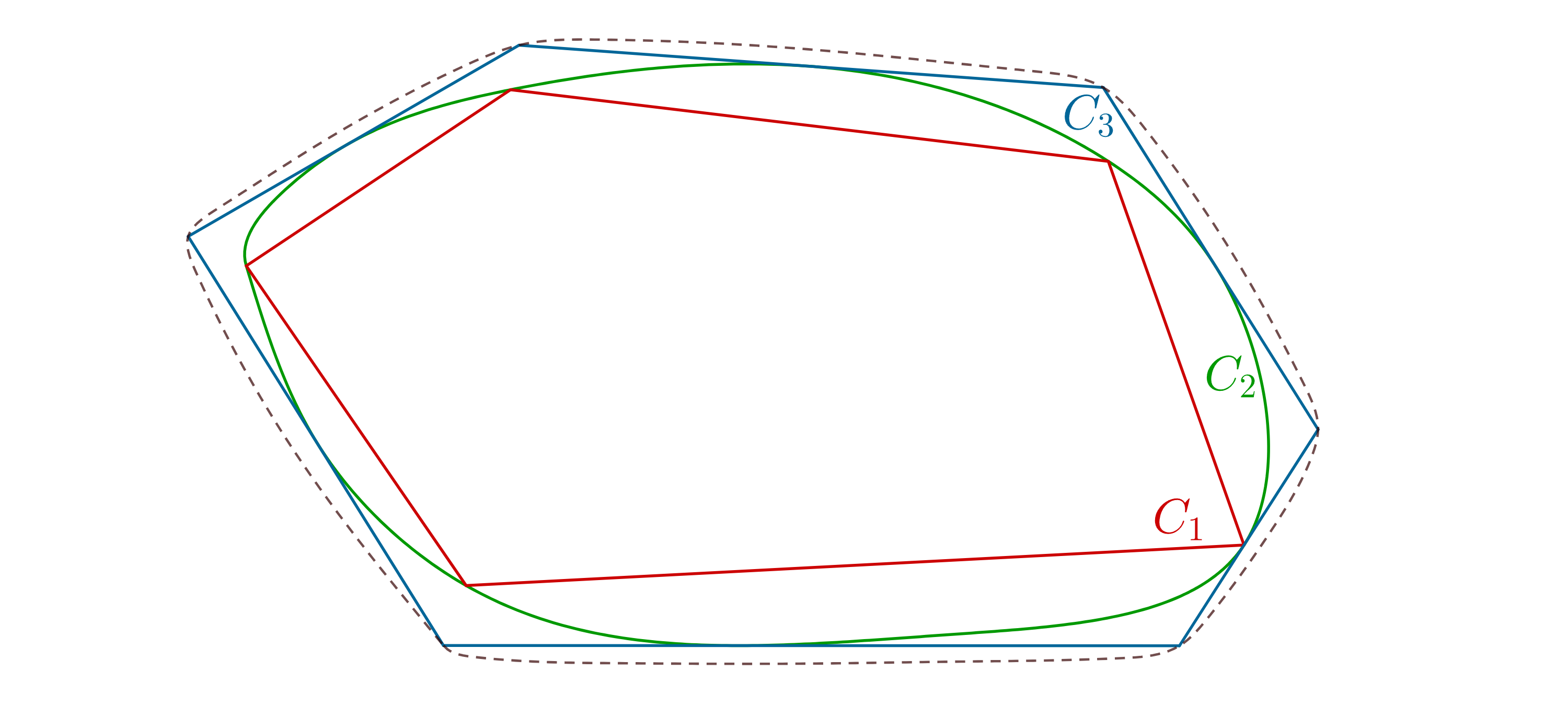}
        \caption{The construction in the proof of Theorem~\ref{main1}.}
    \end{figure}
    
    For the sake of notation, consider the function
    \begin{equation*}
        H_i(N)=\begin{cases}
            \log N &\text{if}\quad \alpha_i=0\\
            N^{1/2} &\text{if}\quad \alpha_i=1/2
        \end{cases}.
    \end{equation*}
    We recall that, by Theorem~\ref{BC} and Theorem~\ref{BT1}, for every $i\in\mathbb{N}$, it holds
    \begin{equation}\label{Estm}
        \inf_{\# \mathcal{P}=N} \mathcal{D}_2(\mathcal{P},\, C_i)\approx_{C_i}H_i(N).
    \end{equation}
    We now construct the sequence $\left\{\eta_i\right\}_{i\in\mathbb{N}}$ recursively. Suppose $\left\{\eta_j\right\}_{j=1}^{i}$ and $\left\{C_j\right\}_{j=1}^{i}$ have already been chosen. Then, let $N_i$ be such that
    \begin{equation}\label{U3}
       H_i^{1-\varepsilon_i}(N)\leq \inf_{\# \mathcal{P}=N} \mathcal{D}_2(\mathcal{P},\, C_i)\leq H_i^{1+\varepsilon_i}(N)\quad\text{for every}\quad N\in[N_i, N_i+q_i].
    \end{equation}
    In particular, notice that such an $N_i$ exists, or else it would contradict \eqref{Estm}. Then, choose $\eta_{i+1}$ in such a way that $\eta_{i+1}\leq\eta_i/2$ and 
    \begin{equation*}
        \eta_{i+1}< (N_i+q_i)^{-4}/16.
    \end{equation*}

    By choosing the whole sequence $\left\{\eta_i\right\}_{i\in\mathbb{N}}$ in such a fashion, it follows that the nested sequence of the planar convex bodies $\left\{C_i\right\}_{i\in\mathbb{N}}$ converges in the Hausdorff metric to a planar convex body $C$ such that, for every $i\in\mathbb{N}$, it holds
    \begin{equation*}
        |C\setminus C_i|\leq\sum_{j=i+1}^\infty\eta_j\leq2\eta_{i+1}.
    \end{equation*}
    Therefore, by applying Lemma~\ref{UseLem}, we obtain that, for every $i\in\mathbb{N}$, for every $N\in[N_i,N_i+q_i]$, and for every configuration $\mathcal{P}$ of $N$ points in $\mathbb{T}^2$, it holds
    \begin{equation*}
        |\mathcal{D}_2(\mathcal{P},\, C)-\mathcal{D}_2(\mathcal{P},\, C_i)|\leq 4N^2\eta_{i+1}^{1/2}<1.
    \end{equation*}
    Finally, by pairing the latter inequality with \eqref{U3}, the claim of Theorem~\ref{main1} follows.

    \subsection{The discrepancy of most convex bodies}
    What follows is inspired by \cite[Sec.~4]{MR1681584}, where the authors applied an analogous argument to study the average decay of the Fourier transform of convex sets. We let $\mathfrak{C}$ be the space of convex bodies in $\mathbb{R}^2$ endowed with the Hausdorff metric $d_H$ defined by
    \begin{equation*}
        d_H(A,B)=\max\Big\{\sup_{x\in A}\inf_{y\in B}|x-y|,\,\sup_{y\in B}\inf_{x\in A}|x-y|\Big\}\quad\text{for }A,B\in\mathfrak{C}.
    \end{equation*}
    By the Blaschke selection theorem, the metric space $(\mathfrak{C}, d_H)$ is locally compact, and by the Baire category theorem, it follows that $(\mathfrak{C}, d_H)$ is a Baire space and, in particular, it is nonmeagre.
    
    In \cite{MR185509} it is shown that the Hausdorff metric and the symmetric-difference metric $d_\Delta(A,B)=|A\Delta B|$ induce the same topology on $\mathfrak{C}$. It is not difficult to adapt the proof of Lemma~\ref{UseLem} and show that, for a fixed $N\in\mathbb{N}$, the functional
    \begin{equation*}
        \mathcal{D}_2(\mathcal{P},\, \cdot)\colon (\mathfrak{C}, d_\Delta)\to (0,+\infty)
    \end{equation*}
    is (locally) uniformly continuous in its argument with respect to all configurations $\mathcal{P}$ of $N$ points. Indeed, it is enough to pair the equality $A\Delta B=(A\setminus B)\sqcup(B\setminus A)$ with the fact that $\mathcal{D}(\mathcal{P},\, \cdot)$ is additive on disjoint sets, and note that the condition on the diameter in the hypothesis of Lemma~\ref{UseLem} is (locally) not restrictive. Hence, we infer that, for a fixed $N\in\mathbb{N}$, the functional
    \begin{equation*}
        \inf_{\# \mathcal{P}=N} \mathcal{D}_2(\mathcal{P},\, \cdot)\colon (\mathfrak{C}, d_\Delta)\to (0,+\infty)
    \end{equation*}
    is continuous. It is time to state a technical result of Gruber~\cite{MR778840}.
    \begin{lem}[Gruber]
        Let $\{a_j\}_{j\in\mathbb{N}}\subset(0,+\infty)$, let $T$ be a nonmeagre topological space, and let $\{\phi_j\}_{j\in\mathbb{N}}$ be a sequence of continuous functionals with $\phi_j\,\colon\, T\to(0,+\infty)$. Define
        \begin{align*}
            \mathcal{A}&=\big\{x\in T\;\colon\; \phi_j(x)=o(a_j)\,\textnormal{ as }\,j\to+\infty\big\},\\
            \mathcal{B}&=\big\{x\in T\;\colon\; a_j=o(\phi_j(x))\,\textnormal{ as }\,j\to+\infty\big\}.
        \end{align*}
        Then, the following holds:
        \begin{enumerate}[label=\textnormal{\roman*)}]
        
            \item If $\mathcal{A}$ is dense in $T$, then, for every $x\in T$ except for a meagre subset, the inequality $\phi_j(x)<a_j$ holds for infinitely many values of $j$.  
            \item If $\mathcal{B}$ is dense in $T$, then, for every $x\in T$ except for a meagre subset, the inequality $\phi_j(x)>a_j$ holds for infinitely many values of $j$.  
        \end{enumerate}
    \end{lem}
    Finally, by pairing Theorem~\ref{BC} and Theorem~\ref{BT1} with the latter lemma and the fact that $\inf_{\# \mathcal{P}=N} \mathcal{D}_2(\mathcal{P},\, \cdot)$ is continuous in $(\mathfrak{C}, d_\Delta)$, we obtain the following result.
    \begin{thm}
        Let $\varepsilon\in(0,1]$. Then, for every planar convex body $C\in\mathfrak{C}$ except for a meagre subset, there exist two infinite subsets $\mathcal N^-_C,\mathcal N^+_C\subset\mathbb{N}$ such that
        \begin{equation*}
            \inf_{\# \mathcal{P}=N} \mathcal{D}_2(\mathcal{P},\, C)\leq \log^{1+\varepsilon}N\quad \forall N\in \mathcal N^-_C \quad \textnormal{ and }\quad \inf_{\# \mathcal{P}=N} \mathcal{D}_2(\mathcal{P},\, C)\geq N^{1/2-\varepsilon}\quad \forall N\in \mathcal N^+_C.
        \end{equation*}
    \end{thm}

    \section{The second method}\label{S2}
The harmonic analysis arguments we are to employ in this section have been recently developed in \cite{MR4358540} and \cite{beretti2025fouriertransformplanarconvex}. These make use of a geometric quantity that we define as follows.\begin{figure}
        \centering
        \includegraphics[width=0.9\linewidth]{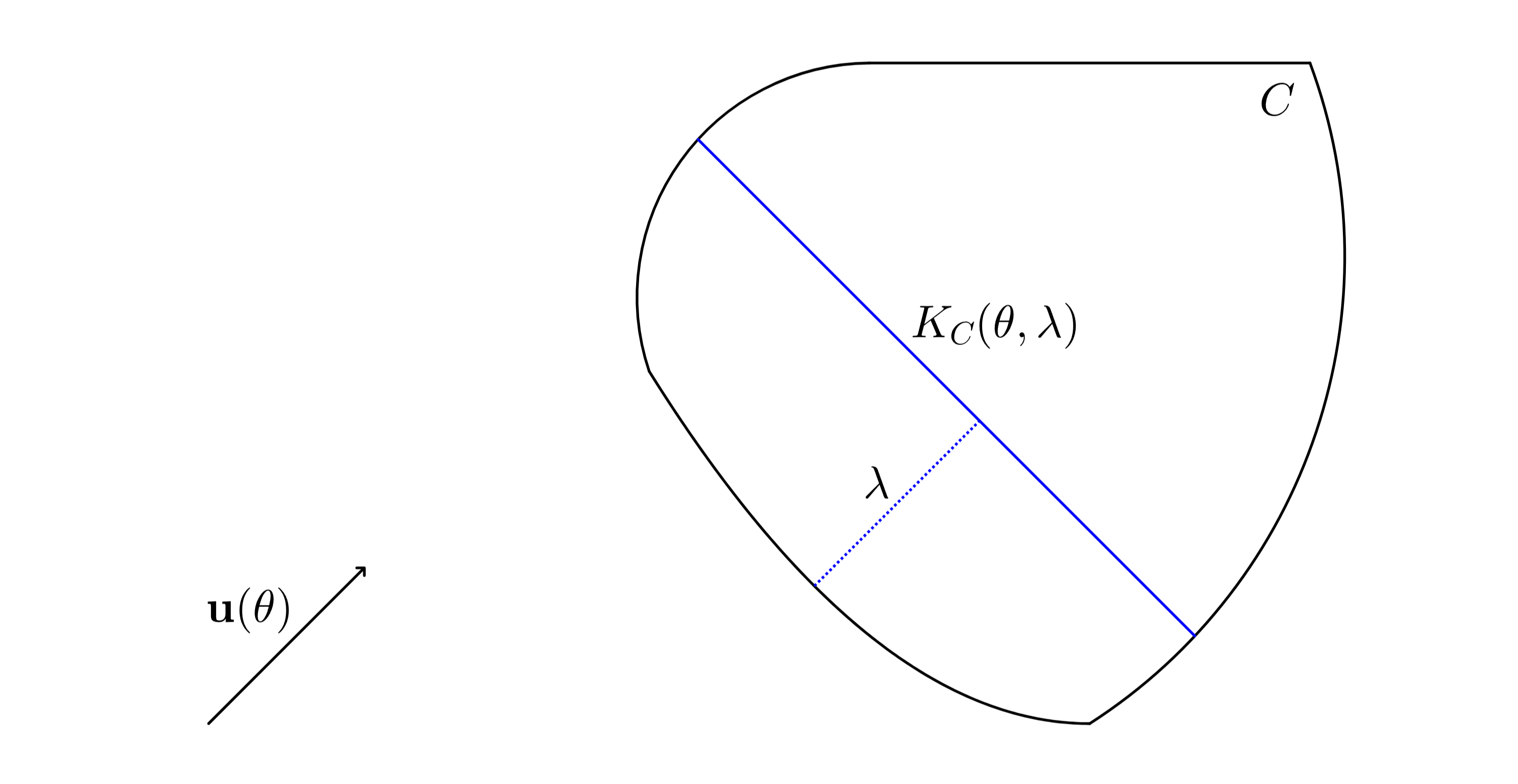}
        \caption{The chord in Definition~\ref{Corde1}.}
    \end{figure}

\begin{defn}\label{Corde1}
    Let $C\subset\mathbb{R}^2$ be a convex body. For an angle $\theta\in[0,2\pi)$ and a value $\lambda>0$, we define the \emph{chord} of $C$ in direction $\uthe=(\cos\theta, \sin\theta)$ at distance $\lambda$ as
		\begin{equation*}
		K_C(\theta,\lambda)=\left\{\mathbf{x}\in C\,\colon\,\mathbf{x}\cdot\uthe=\inf_{\mathbf{y}\in C}(\mathbf{y}\cdot\uthe)+\lambda\right\}.
		\end{equation*}
        Further, we consider its length $\left|K_C(\theta,\lambda)\right|$, and we define the quantity
		\begin{equation*}
		{\gamma}_{C}(\theta,\lambda)=\max\{\left|K_C(\theta,\lambda)\right|,\left|K_C(\theta+\pi,\lambda)\right|\}.
		\end{equation*}
\end{defn}

One may notice that the decay of $|K_C(\theta,\lambda)|$ as $\lambda\to0^+$ depends strongly on the local geometry of the (if just one) point $b\in\partial C$ whose inner normal is $\uthe$. For example, if $\partial C$ is $\mathcal{C}^2$ and has positive curvature at $b$, then $|K_C(\theta,\lambda)|$ ultimately decays as $\lambda^{1/2}$, while if $b$ is the vertex at a corner, then $|K_C(\theta,\lambda)|$ ultimately decays as $\lambda$.

Now, we define the Fourier transform of $\mathds{1}_{C}$ as
\begin{equation*}
    \widehat{\mathds{1}}_C(\boldsymbol{\xi})=\int_{C}e^{-2\pi i \mathbf{x}\cdot\boldsymbol{\xi}}\de \mathbf{x},
\end{equation*}
and state a result that relates averages of the Fourier transform of $\mathds{1}_C$ to the chords of $C$.
\begin{lem}\label{L1}
    Let $C$ be a planar convex body. Then, there exists a positive value $\lambda_C$ such that, for every angle $\theta\in[0,2\pi)$ and every $\lambda\in(0,\lambda_C]$, it holds
    	\begin{equation*}
    	\int_{0}^{1}\left|\widehat{\mathds{1}}_{\delta C}(\lambda^{-1}\,\uthe)\right|^2\de \delta\approx\lambda^{2}\gamma^2_{C}(\theta,\lambda).
    	\end{equation*}
\end{lem}

We refer to \cite[Lem.~3.3]{beretti2025fouriertransformplanarconvex} for a detailed proof of the latter, but we mention that it relies upon the one-dimensional results in \cite{MR1166380} and \cite{MR4358540}.

We now show, through standard calculations, how quadratic discrepancy and the Fourier transform are intertwined. We start by exploiting the convolutional structure of \eqref{Discrepancy1}. By setting
\begin{equation*}
    {\mu}=\sum_{\mathbf{p}\in\mathcal{P}}\mu_{\rm D}(-\mathbf{p})-N\mu_{\rm H},
\end{equation*}
where $\mu_{\rm H}$ is the normalised Haar measure on $\mathbb{T}^2$ and $\mu_{\rm D}(\mathbf{p})$ is the Dirac delta centred at $\mathbf{p}$, one may write
\begin{equation*}
\mathcal{D}(\mathcal{P},\, \boldsymbol{\tau}+ C)=\int_{\mathbb{T}^2}\mathfrak{P}\{\mathds{1}_C\}(-{\boldsymbol{\tau}}-\mathbf{x})\de {\mu}(\mathbf{x})=\left(\mathfrak{P}\{\mathds{1}_C\}\ast{\mu}\right)(-\boldsymbol{\tau}).
\end{equation*}
Now, if $f$ is an integrable function or a finite measure on $\mathbb{T}^2$, we denote by 
\begin{equation*}
    {\mathcal{F}}\{f\}\colon \mathbb{Z}^2\to\mathbb{C}
\end{equation*}
the function of the Fourier coefficients of $f$. In particular, we note that it holds $\mathcal{F}\{{\mu}\}(\mathbf{0})=0$, and moreover, for every $\mathbf{n}\in\mathbb{Z}^2$, it holds
\begin{equation*}
    {\mathcal{F}}\circ \mathfrak{P}\{\mathds{1}_C\}(\mathbf{n})=\widehat{\mathds{1}}_C(\mathbf{n}).
\end{equation*}
Hence, by applying Parseval's identity, we obtain
\begin{align*}
    \int_{\mathbb{T}^2}\left|\mathcal{D}(\mathcal{P},\, \boldsymbol{\tau}+\delta C)\right|^2\de\boldsymbol{\tau}&=\int_{\mathbb{T}^2}\left|(\mathfrak{P}\{\mathds{1}_{\delta C}\} \ast {\mu})\right|^2(\boldsymbol{\tau})\de\boldsymbol{\tau}\\
    &=\sum_{{\mathbf{n}}\in\mathbb{Z}^2}\left|{\mathcal{F}}\circ \mathfrak{P}\{\mathds{1}_{\delta C}\}({\mathbf{n}})\right|^2\left|{\mathcal{F}}\{{\mu}\}({\mathbf{n}})\right|^2\\
    &=\sum_{{\mathbf{n}}\in\mathbb{Z}^2\setminus\{\mathbf{0}\}}\left|\widehat{\mathds{1}}_{\delta C}({\mathbf{n}})\right|^2\left|\sum_{\mathbf{p}\in\mathcal{P}}e^{2 \pi i \mathbf{p}\cdot{\mathbf{n}}}\right|^2.
\end{align*}
Finally, we may rewrite \eqref{Discrepancy2} as
\begin{equation}\label{FourVer}
    \mathcal{D}_2(\mathcal{P},C)=\sum_{{\mathbf{n}}\in\mathbb{Z}^2\setminus\{\mathbf{0}\}}\left|\sum_{\mathbf{p}\in\mathcal{P}}e^{2 \pi i \mathbf{p}\cdot{\mathbf{n}}}\right|^2\int_0^1\left|\widehat{\mathds{1}}_{\delta C}(\mathbf{n})\right|^2\de \delta.
\end{equation}

\subsection{Geometric construction and Fourier estimates}\label{GeoCon}

Recalling the hypotheses of Theorem~\ref{main2}, we have that $\left\{\alpha_i\right\}_{i\in\mathbb{N}}$ is a sequence of exponents with values in $(2/5,1/2)$,  $\left\{\varepsilon_i\right\}_{i\in\mathbb{N}}$ is a (possibly decreasing) sequence in $(0,\varepsilon]$ for some small positive $\varepsilon$, and $\left\{q_i\right\}_{i\in\mathbb{N}}$ is a (possibly increasing) sequence of positive integers. We remark that $\alpha_i$ prescribes the target order of growth, while $\varepsilon_i$ and $q_i$ quantify, respectively, how closely and for how long we approximate it.

We start our construction. Consider a planar convex body $C$ such that it has a centre of symmetry (this ensures $\gamma_C=|K_C|$) and is symmetric with respect to the $y$-axis. Then, suppose that $\mathbf{0}\in\partial C$ and that its inner normal is the vector $(0,1)$. By symmetry, we may construct $\partial C$ in the first quadrant $\R_+^2$. For this purpose, consider the graphs in $\R_+^2$ of the monomials
\begin{equation*}
    \{P_i\}_{i\in\mathbb{N}}\quad\text{with}\quad P_i(x)=x^{\beta_i}\quad\text{and}\quad\beta_i=\frac{2\alpha_i}{2-3\alpha_i}\in(1,2),
\end{equation*}
and notice that their curvatures
\begin{equation}\label{Curvas}
    \kappa_i(x)=\beta_i(\beta_i-1)\frac{x^{\beta_i-2}}{(1+\beta_i^2x^{2\beta_i-2})^{3/2}}
\end{equation}
are strictly decreasing on $(0,+\infty)$, and that $\kappa_i(x)\to+\infty$ as $x\to0^+$.

Now, let $\left\{k_i\right\}_{i\in\mathbb{N}}$ be a sequence to be chosen later, but such that it is contained in $[10,+\infty)$ and it holds $k_{i+1}>k_i$ with $k_i\to+\infty$. By induction, we construct a $\mathcal{C}^2$ curve segment of finite length as follows. First, consider the auxiliary curve segment
\begin{equation*}
    \left\{(x,y)\in\R_+^2\,\colon\,P_1(x)=y,\,a_1<x<b_1,\,\kappa_1(a_1)=k_1,\,\kappa_1(b_1)=10\right\},
\end{equation*}
and then glue at the endpoint $(a_1,P_1(a_1))$ a tangentially-aligned copy of the auxiliary curve segment
\begin{equation*}
    \left\{(x,y)\in\R_+^2\,\colon\,P_2(x)=y,\,a_2<x<b_2,\,\kappa_2(a_2)=k_2,\,\kappa_2(b_2)=k_1\right\}.
\end{equation*}\begin{figure}\centering
        \includegraphics[width=1\linewidth]{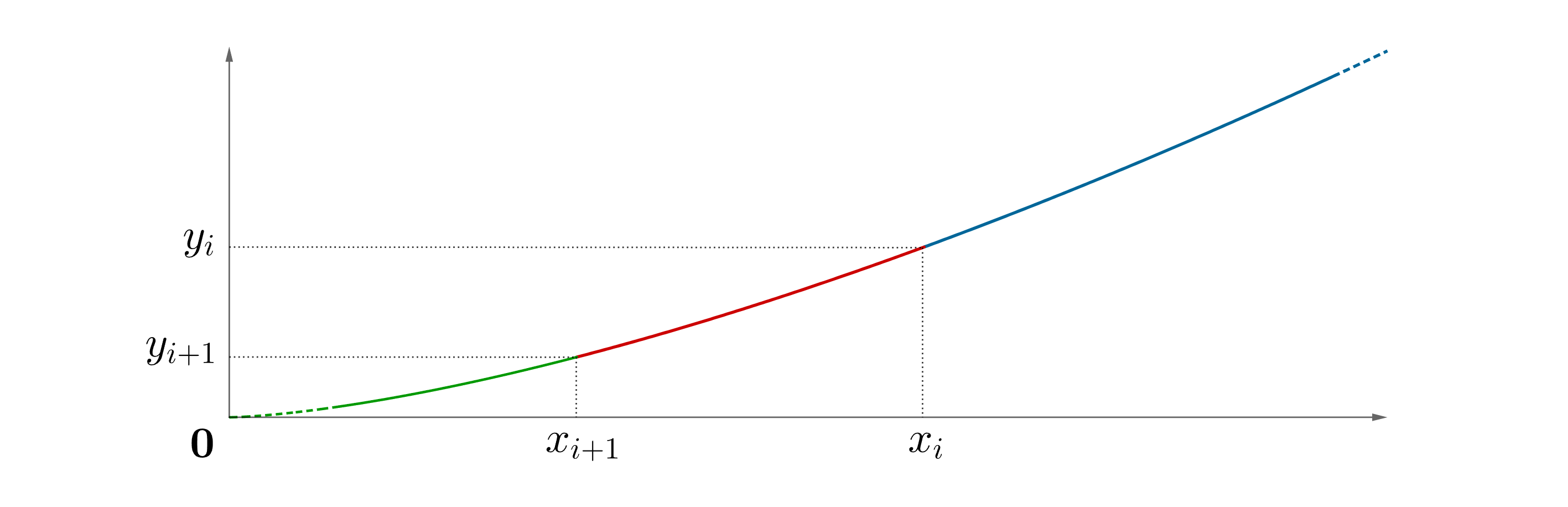}
        \caption{A close-up look at $P$, in which different auxiliary monomial curve segments are displayed in different colours.}
        \label{FF2}
    \end{figure}In particular, notice that we obtain a $\mathcal{C}^2$ link since the signed curvatures at the endpoints $(a_1, P_1(a_1))$ and $(b_2,P_2(b_2))$ coincide. Inductively, and in such a $\mathcal{C}^2$ fashion, we glue all the auxiliary curve segments
\begin{equation*}
    \left\{(x,y)\in\R_+^2\,\colon\,P_i(x)=y,\,a_i<x<b_i,\,\kappa_i(a_i)=k_i,\,\kappa_i(b_i)=k_{i-1}\right\},
\end{equation*}
and in the end, we obtain a $\mathcal{C}^2$ curve segment of finite length. By a final rigid motion, we may place this curve segment so that it is tangent to the $x$-axis and its left endpoint is the origin; see Figure~\ref{FF2} for an illustration of the resulting curve segment $P$. By construction, $P$ has strictly decreasing continuous curvature (and this goes to infinity as one approaches the origin). Finally, we let $\partial C$ coincide with $P$ for its length, and then we design the rest of $\partial C$ in such a way that it is $\mathcal{C}^2$ with signed curvature bounded from above by $\kappa_1(1)$ and from below by, say, $1/10$.

The purpose of the latter construction is the following: $\partial C$ has to look like the graphs of different $P_i$'s at increasingly smaller neighbourhoods of the origin. Ultimately, this is achieved by ensuring that the sequence $\left\{k_i\right\}_{i\in\mathbb{N}}$ increases very rapidly. Indeed, for every $i\in\mathbb{N}$, mark by $(x_i,y_i)$ the left endpoint of $P_i$ after being assembled into $P$, as in Figure~\ref{FF2}. First, the fact that, for $i\in\mathbb{N}$ and $x\in(0,1]$, it holds
\begin{equation*}
    \kappa_i(x)\leq 2x^{-1}\quad\text{and}\quad P'_i(x)\leq 2,
\end{equation*}
ensures that we may get any $(x_i,y_i)$ arbitrarily close to the origin by a large enough choice of $k_i$. Secondly, since the elements in $\left\{k_j\right\}_{j\geq i}$ may be chosen arbitrarily large, we may get the elements in $\left\{P_{j+1}'(b_{j+1})\right\}_{j\geq i}$ arbitrarily close to $0$, and this, in turn, ensures that we may get the inner normal at $(x_i,y_i)$ arbitrarily close to the vector $(0,1)$; in particular, this prevents the curve segment $P$ from becoming a spiral.

In order to apply Lemma~\ref{L1} to $C$, we are interested in how the chords $K_C(\theta+ \pi/2,\lambda)$ behave for small values of $\theta$. For this purpose, let $\beta\in(1,2)$, and consider a second planar convex body $A(\beta)$ whose boundary coincides with the graph of $y=|x|^\beta$ in $[-1,1]^2$. Moreover, suppose that $A(\beta)$ has a centre of symmetry and is symmetric with respect to the $y$-axis. Hence, it holds the following estimates on the chords of $A(\beta)$. 

\begin{prop}\label{Aux}
    Let $A(\beta)$ be as previously defined. Then, there exists a positive small constant $c$ (independent of $\beta$) such that it holds
	\begin{equation*}
	\left|K_{A(\beta)}(\theta+\pi/2,\lambda)\right|\approx_\beta\begin{cases}
		\lambda^{1/\beta} & \textnormal{if}\quad0\leq|\theta|<\lambda^{\frac{\beta-1}{\beta}}< c\\
		\lambda^{1/2}|\theta|^{\frac{2-\beta}{2(\beta-1)}}& \textnormal{if}\quad \lambda^{\frac{\beta-1}{\beta}}\leq |\theta| < c
	\end{cases}.
	\end{equation*}
\end{prop}
We leave the proof of the latter in the Appendix, and we mention that similar calculations had previously been carried out in \cite[Prop.~8.4]{beretti2025fouriertransformplanarconvex}.

Recall that, by a large enough choice of the elements in $\left\{k_j\right\}_{j\geq i}$, we may get $(x_i,y_i)$ arbitrarily close to the origin, as we may get the inner normal at $(x_i,y_i)$ to be arbitrarily close to the vector $(0,1)$. Ultimately, this guarantees that, for every $i\in\mathbb{N}$, we may obtain an interval $[\lambda_{i,1},\lambda_{i,2}]$ and a positive value $\theta_i$ such that, for every $\theta$  such that $|\theta|\leq\theta_i$ and every $\lambda$ such that $\lambda\in\left[\lambda_{i,1},\lambda_{i,2}\right]$, it holds
\begin{equation*}
    \left|K_{C}(\theta+\pi/2,\lambda)\right|\approx \left|K_{A(\beta_i)}(\theta+\pi/2,\lambda)\right|.
\end{equation*}
Notice that, as $\theta_i$ and $\lambda_{i,2}$ depend on $\left\{\beta_j\right\}_{j=1}^{i-1}$ and  $\left\{k_j\right\}_{j=1}^{i-1}$, the choice of $\lambda_{i,1}$ can be made arbitrarily small at every $i$-th step; this will be key for the estimates in the following subsection.

Finally, Proposition~\ref{Aux} and Lemma~\ref{L1} may be applied to $C$ to infer the following result. From now on, we write $\rho^{-1}$ in place of $\lambda$, to prioritise the argument of the Fourier transform.

\begin{lem}\label{L2}
    Let $C$ be as previously defined. Then, one  may choose the sequence $\left\{k_i\right\}_{i\in\mathbb{N}}$ in such a way that, for every $i\in\mathbb{N}$, there exists an interval $[\rho_{i,1},\rho_{i,2}]\subset[1,+\infty)$ (with $\rho_{i+1,1}>\rho_{i,2}$ and $\lim_{i\to+\infty}\rho_{i,1}=+\infty$) and a positive value $\theta_i$ (with $\theta_{i+1}<\theta_i$ and $\lim_{i\to+\infty}\theta_i=0$) such that, for $\rho\in[\rho_{i,1},\rho_{i,2}]$ and $|\theta|\leq\theta_i$, it holds
    \begin{equation*}
	\int_{0}^{1}\left|\widehat{\mathds{1}}_{\delta C}\left(\rho\,\mathbf{u}(\theta+\pi/2)\right)\right|^2\de \delta\approx_{\beta_i}\begin{cases}
		\rho^{-2-2/\beta_i}&\textnormal{if}\quad\left|\theta\right|\leq\rho^{\frac{1-\beta_i}{\beta_i}}\\
		\rho^{-3}\left|\theta\right|^{\frac{2-\beta_i}{\beta_i-1}}&\textnormal{if}\quad\rho^{\frac{1-\beta_i}{\beta_i}}<\left|\theta\right|\leq\theta_i
	\end{cases}.
	\end{equation*}
    In particular, although $\rho_{i,1}$ and $\theta_i$ are determined by the elements in $\left\{k_j\right\}_{j=1}^{i-1}$, the endpoint $\rho_{i,2}$ can be made arbitrarily large upon a larger choice of $k_i$.
\end{lem}
    Starting from the fact that we have designed $P$ so that its curvature is strictly decreasing, we obtain the following auxiliary result.
\begin{lem}\label{Tec1}
    Let $C$ be as previously defined. Then, there exist $\lambda_0>0$ and $\theta_0>\pi/2$ such that, for $\lambda\in[0,\lambda_0)$ and for $\pi/2\leq\theta_1<\theta_2<\theta_0$, it holds
    \begin{equation*}
        |K_C(\theta_1,\lambda)|\leq2|K_C(\theta_2,\lambda)|.
    \end{equation*}
\end{lem}
The proof of the latter requires a technical geometric argument, which we leave to the Appendix. By pairing Lemma~\ref{L1} with Lemma~\ref{Tec1}, we obtain that, for every $\rho>\max\{\lambda_C^{-1},\lambda_0^{-1}\}$ and every angle $\theta_1$ and $\theta_2$ such that $\pi/2\leq\theta_1<\theta_2<\theta_0$, it holds
    \begin{equation}\label{EqX}
        \int_{0}^{1}\left|\widehat{\mathds{1}}_{\delta C}\left(\rho\,\mathbf{u}(\theta_1)\right)\right|^2\de \delta\lesssim\int_{0}^{1}\left|\widehat{\mathds{1}}_{\delta C}\left(\rho\,\mathbf{u}(\theta_2)\right)\right|^2\de \delta.
    \end{equation}

\subsection{Lower bound}
We show how we may construct the sequence $\left\{k_i\right\}_{i\in\mathbb{N}}$ in such a way that, for every $i\in\mathbb{N}$,  it holds the lower bound
\begin{equation*}
    N^{\alpha_i-\varepsilon_i}\leq\inf_{\# \mathcal{P}=N}\mathcal{D}_2(\mathcal{P},\, C)\quad\text{for}\quad N\in[N_i,N_i+q_i].
\end{equation*}
For this purpose, we need an argument of Cassels~\cite{MR0087709} and Montgomery~\cite[Ch.~6]{MR1297543} for estimating exponential sums from below. We formulate it in a version specialised to balls, and the reader may find a short proof either in \cite{MR4358540} or \cite{beretti2025fouriertransformplanarconvex}. In what follows, $B(r)$ stands for the ball centred at the origin with radius $r$.

\begin{lem}[Cassels-Montgomery]\label{C-M}
	For every origin-symmetric convex body $C\subset\mathbb{R}^2$ and for every finite configuration of points $\{\mathbf{p}_j\}_{j=1}^N$ in $\mathbb{T}^2$, it holds
	\begin{equation*}
	\sum_{\mathbf{m}\in(C\setminus B(r))\cap\mathbb{Z}^2}\left|\sum_{j=1}^{N}e^{2\pi i \mathbf{m}\cdot \mathbf{p}_j}\right|^2\geq\frac{|C|}{4}N-2\pi (r^2+1)N^2.
	\end{equation*}
\end{lem}

We proceed with the geometric argument that will pair with the latter lemma. For the sake of notation, we write
\begin{equation*}
    \|\theta\|_\pi=\min_{j\in\mathbb{Z}}|\theta-j\pi|.
\end{equation*}
First, notice that, since we have constructed $\partial C$ in such a way that its curvature is uniformly bounded from below and above when excluding a neighbourhood of the origin (and its symmetric counterpart), then for every $\lambda\in[0,\lambda_0)$ and for every $\theta$ such that $\|\theta\|_\pi>1/10$, it holds
\begin{equation}\label{EqZ}
    |K_C(\theta+\pi/2,\lambda)|\approx\lambda^{1/2}.
\end{equation}
Now, fix an $i\in\mathbb{N}$. Then, by the latter consideration paired with Lemma~\ref{L1}, and by Lemma~\ref{L2} and \eqref{EqX}, we get the lower bound
\begin{equation*}
	\int_{0}^{1}\left|\widehat{\mathds{1}}_{\delta C}\left(\rho\,\mathbf{u}(\theta+\pi/2)\right)\right|^2\de \delta\gtrsim_{\beta_i} \begin{cases}\rho^{-2-2/\beta_i}&\textnormal{if}\quad\rho\in\left[\rho_{i,1},\rho_{i,2}\right]\;\land\;\|\theta\|_\pi\leq1/10\\\rho^{-3}&\textnormal{if}\quad\rho\in\left[\rho_{i,1},\rho_{i,2}\right]\;\land\;\|\theta\|_\pi>1/10\\
		
		0&\textnormal{else}
	\end{cases}.
	\end{equation*}
    
    Let $\tilde\varepsilon_i\in(0,1]$ be a value to be chosen later. Moreover, let $N$, $X$, and $Y$ be positive parameters to be chosen later, but such that $\rho_{i,1}\ll Y\ll X\ll\rho_{i,2}$ and $XY= N^{1+\tilde\varepsilon_i}$. Then, consider a rectangle $R\subset\mathbb{R}^2$ such that it is symmetric with respect to the axes and has a vertex in $(X/2, Y/2)$; consequently, it also holds $|R|=N^{1+\tilde\varepsilon_i}$. Denote by $R(\omega)$ the same rectangle rotated counterclockwise by an angle $\omega$, and define the function $\Phi\colon\mathbb{Z}^2\to[0,+\infty)$ by
	\begin{equation*}
	\Phi(\mathbf{m})=\int_{-1/10}^{1/10}\mathds{1}_{R(\omega)\setminus B(\rho_{i,1})}(\mathbf{m})\de \omega,
	\end{equation*}
	Now, in order to apply the Cassels-Montgomery lemma effectively, we need to find a positive parameter $Z=Z(N)$ such that in the annulus $|\mathbf{m}|\in\left[\rho_{i,1},\rho_{i,2}\right]$ it holds
	\begin{equation*}
	Z\Phi(\mathbf{m})\leq \begin{cases}|\mathbf{m}|^{-2-2/\beta_i}&\textnormal{if}\quad|\mathbf{m}|\in\left[\rho_{i,1},\rho_{i,2}\right]\;\land\;\|\arg(\mathbf{m})+\pi/2\|_\pi\leq1/10\\|\mathbf{m}|^{-3}&\textnormal{if}\quad|\mathbf{m}|\in\left[\rho_{i,1},\rho_{i,2}\right]\;\land\;\|\arg(\mathbf{m})+\pi/2\|_\pi>1/10\\
		
		0&\textnormal{else}
	\end{cases}.
	\end{equation*}
     
    First, we describe two regions where $\Phi$ vanishes. By construction, for all $\mathbf{m}\in \mathbb{Z}^2$  such that $|\mathbf{m}|\in\left[\rho_{i,1},X\right]^\mathsf{c}$, it holds $\Phi(\mathbf{m})=0$. Secondly, for all $\mathbf{m}\in \mathbb{Z}^2$ such that $|\mathbf{m}|\geq Y$ and $\|\arg(\mathbf{m})+\pi/2\|_\pi\leq1/10$, it holds $\Phi(\mathbf{m})=0$ as well.

    Now, for all $\mathbf{m}\in \mathbb{Z}^2$ such that $|\mathbf{m}|\leq Y$, we have the trivial estimate $\Phi(\mathbf{m})\leq1/5$. Since in this region we are aiming for $Z\Phi(\mathbf{m})\leq|\mathbf{m}|^{-2-2/\beta_i}$, then we settle for
     \begin{equation*}
         Z\lesssim Y^{-2-2/\beta_i}.
     \end{equation*}
    Lastly, we consider the remaining region defined by $|\mathbf{m}|\in[Y,X]$ and $\|\arg(\mathbf{m})+\pi/2\|_\pi>1/10$. By some elementary geometry, we find that $\Phi(\mathbf{m})\leq2\psi$ where $\psi$ is such that $|\mathbf{m}|\sin\psi=Y/2$. Hence, it holds $\Phi(\mathbf{m})\leq\pi Y/(2|\mathbf{m}|)$, and since in this region we are aiming for $Z\Phi(\mathbf{m})\leq|\mathbf{m}|^{-3}$, then we settle for
     \begin{equation*}
         Z\lesssim Y^{-1}X^{-2}.
     \end{equation*}
     Having considered all the possible regions, we are left with the requirement
     \begin{equation*}
         Z\lesssim\min\left(Y^{-2-2/\beta_i}, Y^{-1}X^{-2} \right),
     \end{equation*}
    and if, by an optimisation argument, we equalise the two terms in the minimum and use the constraint $XY=N^{1+\tilde\varepsilon_i}$, then we obtain
	\begin{equation*}
	X\approx N^{(1+\tilde\varepsilon_i)(\beta_i+2)/(3\beta_i+2)},\quad Y\approx  N^{(1+\tilde\varepsilon_i)2\beta_i/(3\beta_i+2)},\quad\text{and}\quad Z\approx  N^{-(1+\tilde\varepsilon_i)(4\beta_i+4)/(3\beta_i+2)}.
	\end{equation*}
    
    Starting from \eqref{FourVer}, and by Parseval's identity and the Cassels-Montgomery lemma, we get that, for any configuration of $N$ points $\mathcal{P}=\{\mathbf{p}_j\}_{j=1}^N$ in $\mathbb{T}^2$, it holds
	\begin{align*}
		\int_0^1\int_{\mathbb{T}^2}\left|\mathcal{D}(\mathcal{P},\, \boldsymbol{\tau}+\delta C)\right|^2\de \boldsymbol{\tau}\de\delta&=\sum_{\mathbf{m}\in\mathbb{Z}^2\setminus\{\mathbf{0}\}}\left|\sum_{j=1}^{N}e^{2\pi i \mathbf{m}\cdot \mathbf{p}_j}\right|^2\int_{0}^{1}\left|\widehat{\mathds{1}}_{\delta C}\left(\mathbf{m}\right)\right|^2\de \delta\\
		&\gtrsim_{\beta_i} \sum_{|\mathbf{m}|\geq\rho_{i,1}} \left|\sum_{j=1}^{N}e^{2\pi i \mathbf{m}\cdot \mathbf{p}_j}\right|^2Z\Phi(\mathbf{m})\\
		&= Z\int_{-1/10}^{1/10}\sum_{\mathbf{m}\in R(\omega)\setminus B(\rho_{i,1})}\left|\sum_{j=1}^{N}e^{2\pi i \mathbf{m}\cdot \mathbf{p}_j}\right|^2\de \omega\\
		&\gtrsim Z \left( N^{2+\tilde\varepsilon_i}-2\pi (\rho_{i,1}^2+1)N^2\right).
	\end{align*}
	By construction, $\rho_{i,2}$ can be chosen arbitrarily large, and this also applies to $N$ since our only constraint is $X\ll\rho_{i,2}$. Hence, by taking $\tilde\varepsilon_i$ in such a way that $\tilde\varepsilon_i(\beta_i+2)/(3\beta_i+2)<\varepsilon_i$, we find an $N_i$ such that 
    \begin{equation*}
    N_i^{\tilde\varepsilon_i}\geq4\pi(\rho^2_{i,1}+1)\quad\text{and}\quad (N_i+q_i)^{(1+\tilde\varepsilon_i)(\beta_i+2)/(3\beta_i+2)}\ll \rho_{i,2}.
    \end{equation*}
    Therefore, it follows that, for every $N\in[N_i,N_i+q_i]$, it holds
	\begin{equation*}
	\int_0^1\int_{\mathbb{T}^2}\left|\mathcal{D}(\mathcal{P},\, \boldsymbol{\tau}+\delta C)\right|^2\de \boldsymbol{\tau}\de\delta\gtrsim_{\beta_i} N^{-(1+\tilde\varepsilon_i)(4\beta_i+4)/(3\beta_i+2)}N^{2+\tilde\varepsilon_i}=N^{\frac{2\beta_i}{3\beta_i+2}-\tilde\varepsilon_i\frac{\beta_i+2}{3\beta_i+2}}>N^{\alpha_i-\varepsilon_i}.
	\end{equation*}
    Notice that, since $N^{-\varepsilon_i}$ can be made arbitrarily small, we can ultimately get the implicit constant in front of $N^{\alpha_i-\varepsilon_i}$ to be $1$.

\subsection{Upper bound}

For every $i\in\mathbb{N}$, we proceed to show the upper bound
\begin{equation*}
    \inf_{\# \mathcal{P}=N}\mathcal{D}_2(\mathcal{P},\, C)\lesssim N^{\alpha_i+\varepsilon_i}\quad\text{for}\quad N\in[N_i,N_i+q_i].
\end{equation*}
This is done by employing suitable sampling sequences, and we will make use of the kind that appear in \cite[Thm.~8]{MR4358540}.

Fix $i\in\mathbb{N}$, and again, we recall that at every $i$-th step of our construction, $\rho_{i,2}$ serves as a free parameter that may be chosen arbitrarily large. By pairing Lemma~\ref{L1} with \eqref{EqZ}, and by Lemma~\ref{L2} and \eqref{EqX}, in the annulus $\rho\in[\rho_{i,1},\rho_{i,2}]$ it holds the upper bound
\begin{equation}\label{Est2}
	\int_{0}^{1}\left|\widehat{\mathds{1}}_{\delta C}\left(\rho\,\mathbf{u}(\theta+\pi/2)\right)\right|^2\de \delta\lesssim_{\beta_i}\begin{cases}
		\rho^{-2-2/\beta_i}&\textnormal{if}\quad\|\theta\|_\pi\leq\rho^{\frac{1-\beta_i}{\beta_i}}\\
		\rho^{-3}\|\theta\|_\pi^{\frac{2-\beta_i}{\beta_i-1}}&\textnormal{if}\quad\rho^{\frac{1-\beta_i}{\beta_i}}<\|\theta\|_\pi\leq\theta_i\\
        \rho^{-3}&\textnormal{if}\quad\theta_i<\|\theta\|_\pi
        
	\end{cases}.
	\end{equation}
    By Lemma~\ref{L1}, and since we have constructed $\partial C$ in such a way that its curvature is uniformly bounded from below, then, for every $\rho>\lambda_C^{-1}$ and any angle $\theta$, it holds
    \begin{equation}\label{Est3}
        \int_{0}^{1}\left|\widehat{\mathds{1}}_{\delta C}\left(\rho\,\mathbf{u}(\theta)\right)\right|^2\de \delta\lesssim \rho^{-3}.
    \end{equation}
    Again, let $\tilde N\in\mathbb{N}$ be a free large parameter to be chosen later, and set
    \begin{equation*}
        G=\lfloor \tilde N^{\frac{2+\beta_i}{2+3\beta_i}}\rfloor,\quad L=\lfloor \tilde N^{\frac{2\beta_i}{2+3\beta_i}}\rfloor,\quad J_G=[0,G-1]\cap\mathbb{N},\quad\text{and}\quad J_L=[0,L-1]\cap\mathbb{N}. 
    \end{equation*}
    Set $N=GL$ (trivially, $L<G$ and $N/2<\tilde N<2N$), and consider the configuration of $N$ points $\mathcal{P}$ in $\mathbb{T}^2$ defined by
	\begin{equation*}
	\mathcal{P}=\left\{\mathbf{p}_j\right\}_{j=1}^N=\{\mathbf{p}_{\ell,g}\}_{g\in J_G,\,\ell\in J_L}\quad\text{with}\quad \mathbf{p}_{\ell,g}=\left(\frac{g}{G},\,\frac{\ell}{L}\right).
	\end{equation*}
	Recall that, by Parseval's identity, we get
	\begin{equation*}
	\int_{\mathbb{T}^2}\left|\mathcal{D}(\mathcal{P}, \boldsymbol{\tau}+\delta C)\right|^2\de \boldsymbol{\tau}=\sum_{\mathbf{m}\neq(0,0)}\left|\widehat{\mathds{1}}_{\delta C}(\mathbf{m})\right|^2\left|\sum_{g\in J_G}\sum_{\ell\in J_L}e^{2\pi i \mathbf{m}\cdot \mathbf{p}_{\ell,g}}\right|^2,
	\end{equation*}
	and in particular, it holds the identity
	\begin{equation*}
	\sum_{g\in J_G}\sum_{\ell\in J_L}e^{2\pi i\left(m_1\frac{g}{G}+m_2\frac{\ell}{L}\right)}=
	\begin{cases}
	GL &\text{if}\quad m_1\in G\mathbb{Z}\quad\text{and}\quad m_2\in L\mathbb{Z}\\
		0 &\text{else}
	\end{cases}.
	\end{equation*}
	Hence, consider the set $\mathcal{R}=(G\mathbb{Z}\times L\mathbb{Z})\setminus\{\mathbf{0}\}$, and split it into the regions
	\begin{equation*}
	\begin{split}
	V_1&=\left\{\mathbf{m}\in\mathcal{R}\,\colon\,|m_1|\leq|m_2|^{1/\beta_i}\right\},\\
	V_2&=\left\{\mathbf{m}\in\mathcal{R}\,\colon\,|m_2|^{1/\beta_i}<|m_1|\leq\theta_i|m_2|\right\},\\
	V_3&=\left\{\mathbf{m}\in\mathcal{R}\,\colon\,\theta_i|m_2|<|m_1|\right\},
	\end{split}
	\end{equation*}
	so that we may write
	\begin{equation}\label{e7}
		\mathcal{D}_2(\mathcal{P},\, C)=G^2L^2\left(\sum_{\mathbf{m}\in V_1}+\sum_{\mathbf{m}\in V_2}+\sum_{\mathbf{m}\in V_3}\right)\int_{0}^{1}\left|\widehat{\mathds{1}}_{\delta C}(\mathbf{m})\right|^2\de\delta.
	\end{equation}
        Now, we constrain $\tilde N$ to be so large that $L>\rho_{i,1}$, and as a consequence, it follows that $\mathcal{R}\cap B(\rho_{i,1})=\varnothing$. Moreover, we constrain $\rho_{i,2}$ to be so large that
        \begin{equation*}
            \mathcal{Q}_i=\sum_{\mathbf{m}\in\mathcal{R}\cap B^\mathsf{c}(\rho_{i,2})}|\mathbf{m}|^{-3}< N^{-2}.
        \end{equation*}
        We exploit \eqref{Est2} to study the three sums in \eqref{e7}. In this case, we must consider
        \begin{equation*}
            \rho=|\mathbf{m}|\quad\text{and}\quad\tan\theta=-\frac{m_1}{m_2}.
        \end{equation*}
        For the sake of notation, we write the annulus $\mathcal{A}=\left\{\mathbf{m}\in\mathcal{R}\,\colon\,|\mathbf{m}|\in[\rho_{i,1},\rho_{i,2}]\right\}$. Notice that for small $\theta$ it holds $\tan\theta\approx\theta$, and consequently, with a bit of rearrangement, we can rewrite the estimates in \eqref{Est2} as: for every $\mathbf{m}\in \mathcal{A}$ it holds
        \begin{equation}\label{e5-2}
	\int_{0}^{1}\left|\widehat{\mathds{1}}_{\delta C}\left(\mathbf{m}\right)\right|^2\de \delta\lesssim_{\beta_i}\begin{cases}
		|m_2|^{-2-2/\beta_i}&\text{if}\quad|m_1|\leq|m_2|^{1/\beta_i}\\
		|m_1|^{\frac{2-\beta_i}{\beta_i-1}}|m_2|^{\frac{1-2\beta_i}{\beta_i-1}}&\text{if}\quad|m_2|^{1/\beta_i}\leq|m_1|<\theta_i|m_2|\\
        |m_1|^{-3}
        &\text{if}\quad\theta_i|m_2|<|m_1|
	\end{cases}.
	\end{equation}
    For $j\in\{1,2,3\}$, notice that by the constraints we have imposed earlier, it holds $V_j\cap\mathcal{A}=V_j\cap B(\rho_{i,2})$ and $G^2L^2\mathcal{Q}_i=N^2\mathcal{Q}_i<1$, and by further applying \eqref{Est3}, we obtain
    \begin{equation}\label{Est4}
    \begin{split}
        G^2L^2\sum_{\mathbf{m}\in V_j}\int_{0}^{1}\left|\widehat{\mathds{1}}_{\delta C}(\mathbf{m})\right|^2\de\delta&=G^2L^2\left(\sum_{\mathbf{m}\in V_j\cap\mathcal{A}}+\sum_{\mathbf{m}\in V_j\cap B^{\mathsf{c}}(\rho_{i,2})}\right)\int_{0}^{1}\left|\widehat{\mathds{1}}_{\delta C}(\mathbf{m})\right|^2\de\delta\\&\lesssim G^2L^2\sum_{\mathbf{m}\in V_j\cap\mathcal{A}}\int_{0}^{1}\left|\widehat{\mathds{1}}_{\delta C}(\mathbf{m})\right|^2\de\delta+G^2L^2\mathcal{Q}_i\\ &\lesssim G^2L^2\sum_{\mathbf{m}\in V_j\cap\mathcal{A}}\int_{0}^{1}\left|\widehat{\mathds{1}}_{\delta C}(\mathbf{m})\right|^2\de\delta.
        \end{split}
    \end{equation}
    Then, by applying \eqref{e5-2} and \eqref{Est4}, for the first sum in the last term of \eqref{e7} we obtain
	\begin{align*}
		\sum_{\mathbf{m}\in V_1} \int_{0}^{1}\left|\widehat{\mathds{1}}_{ \delta C}(\mathbf{m})\right|^2\de\delta&\lesssim_{\beta_i} \sum_{\mathbf{m}\in V_1}|m_2|^{-2-2/\beta_i}\lesssim \sum_{n_2=1}^{+\infty}\,\sum_{n_1=0}^{(Ln_2)^{1/\beta_i}/G}(Ln_2)^{-2-2/\beta_i}\\
        &\lesssim  L^{-2-2/\beta_i}\sum_{n_2=1}^{+\infty}n_2^{-2-2/\beta_i}\left(1+(Ln_2)^{1/\beta_i}/G\right)\\ &\lesssim L^{-2-2/\beta_i}.
	\end{align*}
        Similarly, for the second sum in the last term of \eqref{e7}, we get
    \begin{align*}
		\sum_{\mathbf{m}\in V_2} \int_{0}^{1}\left|\widehat{\mathds{1}}_{ \delta C}(\mathbf{m})\right|^2\de\delta&\lesssim_{\beta_i} \sum_{\mathbf{m}\in V_2}|m_1|^{\frac{2-\beta_i}{\beta_i-1}}|m_2|^{\frac{1-2\beta_i}{\beta_i-1}}\\
		&\lesssim \sum_{n_1=1}^{+\infty}\,\sum_{n_2=Gn_1/(\theta_iL)}^{(Gn_1)^{\beta_i}/L}(Gn_1)^{\frac{2-\beta_i}{\beta_i-1}}(Ln_2)^{\frac{1-2\beta_i}{\beta_i-1}}\\
		&\lesssim G^{\frac{2-\beta_i}{\beta_i-1}} L^{\frac{1-2\beta_i}{\beta_i-1}}\sum_{n_1=1}^{+\infty}n_1^{\frac{2-\beta_i}{\beta_i-1}}\sum_{n_2=Gn_1/(\theta_iL)}^{+\infty}n_2^{\frac{1-2\beta_i}{\beta_i-1}}\\
		&\lesssim G^{\frac{2-\beta_i}{\beta_i-1}} L^{\frac{1-2\beta_i}{\beta_i-1}}\sum_{n_1=1}^{+\infty}n_1^{\frac{2-\beta_i}{\beta_i-1}}\left(Gn_1/(\theta_iL)\right)^{\frac{-\beta_i}{\beta_i-1}}\\
        &\lesssim G^{-2}L^{-1}\sum_{n_1=1}^{+\infty}n_1^{-2}\lesssim G^{-2}L^{-1}.
	\end{align*}
	For the last sum in the last term in \eqref{e7}, we obtain
	\begin{align*}
		\sum_{\mathbf{m}\in V_3} \int_{0}^{1}\left|\widehat{\mathds{1}}_{ \delta C}(\mathbf{m})\right|^2\de\delta&\lesssim_{\beta_i} \sum_{\mathbf{m}\in V_3}|m_1|^{-3}\lesssim \sum_{n_1=1}^{+\infty}\,\sum_{n_2=0}^{Gn_1/(\theta_i L)}(Gn_1)^{-3}\\
        &\lesssim G^{-3}\sum_{n_1=1}^{+\infty}n_1^{-2}G/(\theta_i L)\lesssim G^{-2}L^{-1}\theta_i^{-1}.
	\end{align*}
    Finally, since $G^2L^{-2/\beta_i}=L\approx N^{\frac{2\beta_i}{2+3\beta_i}}=N^{\alpha_i}$, and by comparing the previous three estimates with \eqref{e7},  we obtain
    \begin{equation*}
        \mathcal{D}_2(\mathcal{P},\, C)\lesssim_{\beta_i} G^2L^2\left(L^{-2-2/\beta_i}+G^{-2}L^{-1}/\theta_i\right)\lesssim G^2L^{-2/\beta_i}+L/\theta_i\lesssim N^{\alpha_i}\theta_i^{-1}
    \end{equation*}
    Hence, since we may choose $\rho_{i,2}$ and $\tilde N$ arbitrarily large, then we can ultimately obtain $N$ such that $N^{\varepsilon_i}\geq\theta_i^{-1}$, and for the same reason, we may ultimately write
    \begin{equation*}
        \mathcal{D}_2(\mathcal{P},\, C)\leq N^{\alpha_i+\varepsilon_i}.
    \end{equation*}
    
    At this stage, by the constraints we have imposed, we may only conclude that the upper bound holds for all the integers $N$ of the form
    \begin{equation*}
        N=\lfloor n^{\frac{2+\beta_i}{2+3\beta_i}}\rfloor\,\lfloor n^{\frac{2\beta_i}{2+3\beta_i}}\rfloor,\quad\text{with}\quad n^{\alpha_i}>\rho_{i,1}\quad\text{and}\quad n^2<\rho_{i,2}.
    \end{equation*}
    In order to prove that there is a suitable choice of points for some $q_i$ consecutive integers in the interval $(\rho_{i,1}^{1/\alpha_i},\rho_{i,2}^{1/2})$, consider the following recursive definition
	\begin{equation*}
	n_j=\max\left\{n\in\mathbb{N}\;\colon\lfloor n^{1-\alpha_i}\rfloor\,\lfloor n^{\alpha_i}\rfloor \leq N-\sum_{k=1}^{j-1}\,\lfloor n_k^{1-\alpha_i}\rfloor\,\lfloor n_k^{\alpha_i}\rfloor\right\}\quad\text{for}\quad j\in\mathbb{N}\setminus\{0\},
	\end{equation*}
	where improper sums are conventionally considered as zeros.
	Now, notice that the latter definition implies that
	\begin{equation*}
 \begin{split}
	N-\sum_{k=1}^{j-1}\,\lfloor n_k^{1-\alpha_i}\rfloor\,\lfloor n_k^{\alpha_i}\rfloor&\leq\lfloor (n_j+1)^{1-\alpha_i}\rfloor\,\lfloor (n_j+1)^{\alpha_i}\rfloor\\
    &\leq\lfloor n_j^{1-\alpha_i}\rfloor\,\lfloor n_j^{\alpha_i}\rfloor+\lfloor n_j^{1-\alpha_i}\rfloor+\lfloor n_j^{\alpha_i}\rfloor+1,
  \end{split}
	\end{equation*}
	and therefore, it follows that
	\begin{equation*}
	N-\sum_{k=1}^{j}\,\lfloor n_k^{1-\alpha_i}\rfloor\,\lfloor n_k^{\alpha_i}\rfloor\leq2n_j^{1-\alpha_i}.
	\end{equation*}
	Again, by the definition of $n_j$, we point out that
	\begin{equation*}
	\frac{n_1}{2}\leq N\quad\text{and}\quad
	\frac{n_{j+1}}{2}\leq N-\sum_{k=1}^{j}\,\lfloor n_k^{1-\alpha_i}\rfloor\,\lfloor n_k^{\alpha_i}\rfloor,
	\end{equation*}
	and therefore, by induction, we obtain
	\begin{equation*}
	N-\sum_{k=1}^{j}\,\lfloor n_k^{1-\alpha_i}\rfloor\,\lfloor n_k^{\alpha_i}\rfloor\leq2n_j^{1-\alpha_i}\leq2^{4}n_{j-1}^{(1-\alpha_i)^2}\leq2^{2j}N^{\left(1-\alpha_i\right)^j}.
	\end{equation*}
	In particular, notice that
	\begin{equation*}
	N-\sum_{k=1}^{5}\lfloor n_k^{1-\alpha_i}\rfloor\,\lfloor n_k^{\alpha_i}\rfloor\leq2^{10} N^{\left(1-\alpha_i\right)^{5}}\ll\bigl(N^{\alpha_i/2}\bigr).
	\end{equation*}
	Then, if we further impose that $N^{\alpha_i^2/2}>\rho_{i,1}$, we may associate a choice of points as in the previous construction to every $N_j=\lfloor n_j^{1-\alpha_i}\rfloor\,\lfloor n_j^{\alpha_i}\rfloor$ such that $N_j\gtrsim N^{\alpha_i/2}$, and we do not bother the contribution of the remaining points since, by the latter equation, it is certainly much smaller than $N^{\alpha_i}$. In particular, it is not restrictive to admit repetition of points in the sampling sequence $\mathcal{P}_N$. Finally, it is harmless to sum the contribution of our (at most) five point-configurations and the remainder points, since, for $a_1,\ldots,a_6\geq0$, it holds
    \begin{equation*}
	\left(\sum_{j=1}^{6}a_j\right)^2\leq6\sum_{j=1}^{6}a_j^2.
	\end{equation*}

    By the argument we have just carried over, we may find an integer $N_i\gg q_i$ for which the upper bound in Theorem~\ref{main2} holds, and by choosing $\rho_{i,2}$ large enough, we ensure that the same construction may be carried over for the next consecutive $q_i$ integers after $N_i$.

    \begin{rem}
        Notice that all the constraints needed in the proof of the lower bound and in the proof of the upper bound are ultimately a matter of choosing $\rho_{i,2}$ large enough; this guarantees that we may get the $N_i$ obtained at the end of the two proofs to be the same integer.
    \end{rem}

\appendix
    \section{Auxiliary results}

We start by proving the technical lemma at the end of Subsection~\ref{GeoCon}.
\begin{proof}[Proof of Lemma~\ref{Tec1}]\begin{figure}
        \centering
        \includegraphics[width=0.9\linewidth]{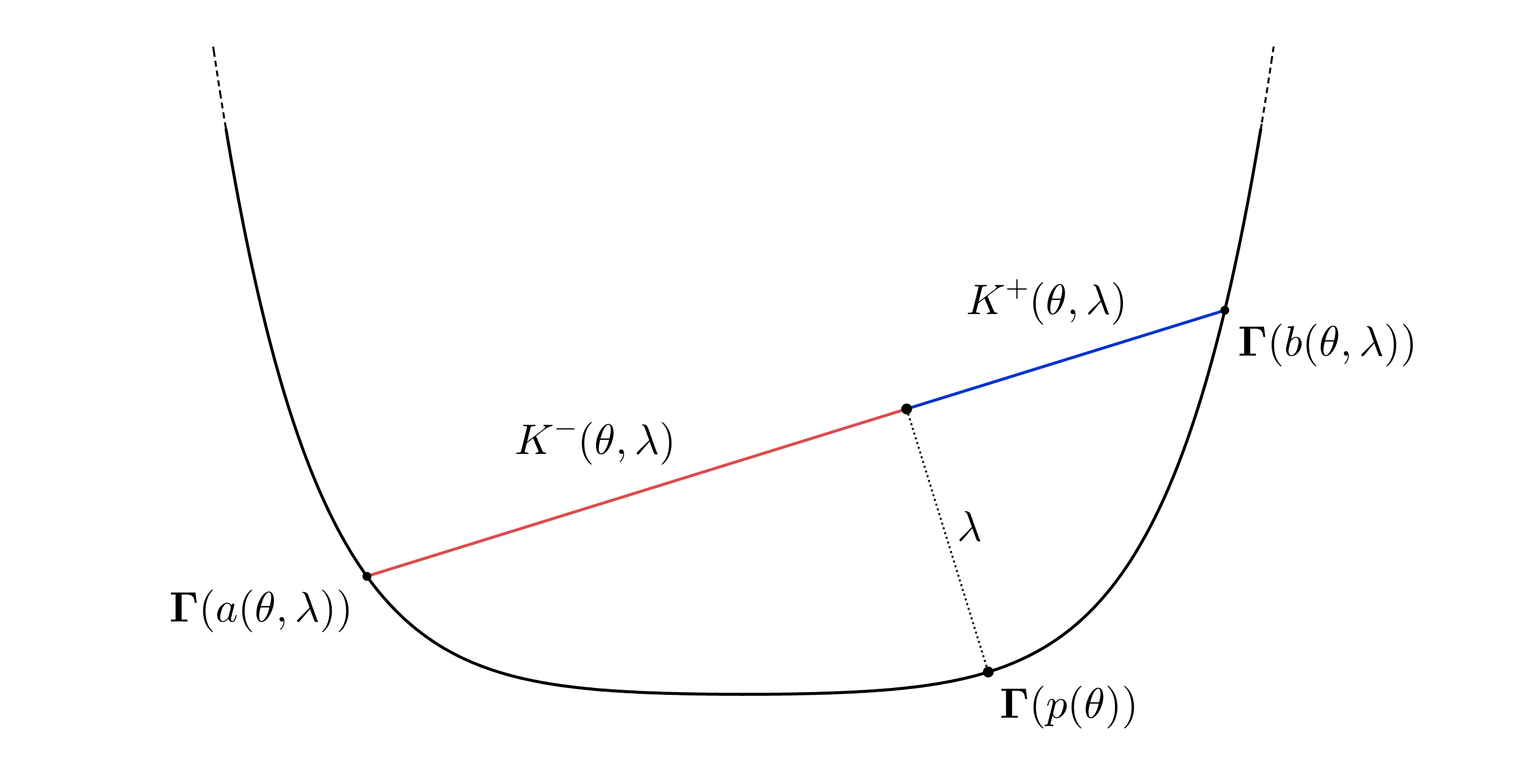}
        \caption{The auxiliary tools we constructed.}
        \label{FigAux}
    \end{figure}
    We need some auxiliary geometric tools to prove Lemma~\ref{Tec1}, and for the sake of notation, we omit $C$ in the subscript. Without loss of generality, we let $\boldsymbol{\Gamma}\colon[-1,1)\to\R^2$ be the counter-clockwise arc-length parametrisation of $\partial C$ such that $\boldsymbol{\Gamma}(0)=\mathbf{0}$, and we write $\boldsymbol{\Gamma}'$ as its velocity (in particular, $\left|\boldsymbol{\Gamma}'\right|\equiv1$).
    For a chord $K(\theta,\lambda)$, we define $a(\theta,\lambda)$ and $b(\theta,\lambda)$ to be the parametrisation by $\boldsymbol{\Gamma}$ of the extreme points of $K(\theta,\lambda)$, with the convention that
    \begin{equation*}
        \boldsymbol{\Gamma}(a(\theta,\lambda))-\boldsymbol{\Gamma}(b(\theta,\lambda))=|K(\theta,\lambda)|\mathbf{u}'(\theta).
    \end{equation*}
    Moreover, we set $p(\theta)$ to be the parametrisation by $\boldsymbol{\Gamma}$ of the point at $\partial C$ whose inner normal is $\uthe$. Finally, we split the chord $K(\theta,\lambda)$ into two semi-chords $K^+(\theta,\lambda)$ and $K^-(\theta,\lambda)$ as displayed in Figure~\ref{FigAux}. Now, choose $\theta_0$ and $\lambda_0$ so that, for $\lambda\in[0,\lambda_0)$ and $\theta\in[\pi/2,\theta_0)$, it holds $\boldsymbol{\Gamma}_2'(b(\theta,\lambda))\leq1/10$. The latter requirement bounds the endpoints of the chords $K(\theta,\lambda)$ to the curve segment constructed in Subsection~\ref{GeoCon}. Thus, we may conclude once we prove that, for $\lambda\in[0,\lambda_0)$ and $\pi/2\leq\theta_1<\theta_2<\theta_0$, it holds
    \begin{equation*}
        \left|K^+(\theta_2,\lambda)\right|\geq\left|K^+(\theta_1,\lambda)\right|\geq\left|K^-(\theta_1,\lambda)\right|.
    \end{equation*}
    We only show the inequality $\left|K^+(\theta_2,\lambda)\right|\geq\left|K^+(\theta_1,\lambda)\right|$, since the argument for the other one is analogous. We begin by noting that, for $i\in\{1,2\}$, it holds
    \begin{equation}\label{Inu1}
        \lambda=\int_{0}^{b(\theta_i,\lambda)-p(\theta_i)}\mathbf{u}(\theta_i)\cdot\boldsymbol{\Gamma}'(p(\theta_i)+s)\de s
    \end{equation}
    and
    \begin{equation}\label{Inu2}
        \left|K^+(\theta_i,\lambda)\right|=\int_{0}^{b(\theta_i,\lambda)-p(\theta_i)}\mathbf{u}(\theta_i-\pi/2)\cdot\boldsymbol{\Gamma}'(p(\theta_i)+s)\de s.
    \end{equation}
    Since the derivative of the argument of $\boldsymbol{\Gamma}'$ is exactly the signed curvature $\kappa$, then for $i\in\{1,2\}$ we obtain
    \begin{equation*}
        \boldsymbol{\Gamma}'(p(\theta_i))=\mathbf{u}(\theta_i-\pi/2)\quad\text{and}\quad\boldsymbol{\Gamma}'(p(\theta_i)+s)=\mathbf{u}\left(\theta_i-\pi/2+\int_0^s\kappa(p(\theta_i)+t)\de t\right),
    \end{equation*}
    so that, by some basic trigonometry, we may rewrite \eqref{Inu1} and \eqref{Inu2} as
    \begin{equation}\label{Inu3}
        \lambda=\int_{0}^{b(\theta_i,\lambda)-p(\theta_i)}\sin\left(\int_0^s\kappa(p(\theta_i)+t)\de t\right)\de s
    \end{equation}
    and
    \begin{equation}\label{Inu4}
        \left|K^+(\theta_i,\lambda)\right|=\int_{0}^{b(\theta_i,\lambda)-p(\theta_i)}\cos\left(\int_0^s\kappa(p(\theta_i)+t)\de t\right)\de s.
    \end{equation}
    Notice that, by our construction of the curve segment $P$ in Subsection~\ref{GeoCon}, it holds the inequality
    \begin{equation*}
        \kappa(p(\theta_1)+t)\geq\kappa(p(\theta_2)+t)\quad\text{for}\quad t\in\left[0, b(\theta_0,\lambda_0)-p(\theta_2)\right),
    \end{equation*}
    but this, combined with \eqref{Inu3} and the monotonicity of the sine function in $[0,\pi/2)$, implies that
    \begin{equation*}
        b(\theta_1,\lambda)-p(\theta_1)\leq b(\theta_2,\lambda)-p(\theta_2),
    \end{equation*}
    and in turn, combining the latter with \eqref{Inu4} and the monotonicity of the cosine function in $[0,\pi/2)$, implies
    \begin{equation*}
        \left|K^+(\theta_2,\lambda)\right|\geq\left|K^+(\theta_1,\lambda)\right|.
    \end{equation*}
    \end{proof}

    Before proceeding with the proof of Proposition~\ref{Aux}, we first state the following auxiliary result. We will not present its proof here, as it is a simple exercise; however, the reader may find a proof in \cite[Lemma~8.1]{beretti2025fouriertransformplanarconvex}.
    \begin{lem}\label{r2}

    Let $\alpha$ and $\beta$ be positive numbers, and let $g\colon[0,+\infty)\to[0,+\infty)$ be such that
    \begin{equation*}
    g(x)\approx\begin{cases}
	   x^\alpha& \textnormal{if}\quad 0\leq x< 1\\
	   x^\beta& \textnormal{if}\quad x\geq 1 
    \end{cases}.
    \end{equation*}
    If $x_y$ is such that $g(x_y)=y$, then it holds
    \begin{equation*}
    x_y\approx\begin{cases}
	   y^{1/\alpha} & \textnormal{if}\quad 0\leq y<1 \\
	   y^{1/\beta} & \textnormal{if}\quad y\geq1
    \end{cases}.
    \end{equation*}
    \end{lem}

    We proceed with the proof of Proposition~\ref{Aux}, and by symmetry, we may restrict ourselves to the case $\theta\geq0$. Since the curvature of the graph of $y=|x|^\beta$ decreases as one moves away from the origin, then, by arguing as in the proof of Lemma~\ref{Tec1}, we obtain that there exist positive small values $\tilde\lambda$ and $\tilde\theta$ such that, for every $\lambda\in[0,\tilde\lambda)$ and $\theta\in[0,\tilde\theta)$, it holds
    \begin{equation*}
        \left|K^+_{A(\beta)}(\theta+\pi/2,\lambda)\right|\geq \left|K^-_{A(\beta)}(\theta+\pi/2,\lambda)\right|.
    \end{equation*}
    In this same range, the endpoints of the chord $K(\theta+\pi/2,\lambda)$ are bounded to the part of $\partial A (\beta)$ that coincides with the graph of $y=|x|^\beta$. Thus, it is enough to establish the estimates in Proposition~\ref{Aux} for $K^+_{A(\beta)}(\theta+\pi/2,\lambda)$ with $\lambda\in[0,\tilde\lambda)$ and $\theta\in[0,\tilde\theta)$.
    
    \begin{proof}[Proof of Proposition~\ref{Aux}]
    Without loss of generality, let $\boldsymbol{\Gamma}_{A(\beta)}\colon[-1,1)\to\R^2$ be the counter-clockwise arc-length parametrisation of $\partial A(\beta)$ such that $\boldsymbol{\Gamma}_{A(\beta)}(0)=\mathbf{0}$, and let $b$ and $p$ be defined analogously as in the proof of Lemma~\ref{Tec1}. For the sake of notation, we set $x_o(\theta)$ to be the abscissa of $\boldsymbol{\Gamma}_{A(\beta)}(p(\theta))$ and $x_+(\theta,\lambda)$  to be the abscissa of $\boldsymbol{\Gamma}_{A(\beta)}(b(\theta,\lambda))$. First, by some basic geometric calculations, we note that, for $\lambda\in[0,\tilde\lambda)$ and for $\theta\in[0,\tilde\theta)$, it holds
	\begin{equation*}
	\left|K_{A(\beta)}^+(\theta+\pi/2,\lambda)\right|\approx x_+(\theta+\pi/2,\lambda)-x_o(\theta+\pi/2).
	\end{equation*}
    For the sake of brevity, we will write $x_o=x_o(\theta+\pi/2)$ and $x_+=x_+(\theta+\pi/2,\lambda)$. Now, notice that $x_+$ is the abscissa of the (right) intersection of the curve $y=|x|^\beta$ with the straight line
	\begin{equation*}y=(x-x_o)\beta x_o^{\beta-1}+x_o^\beta+\frac{\lambda}{\cos\theta}.
	\end{equation*}
    Equating, and with the normalisation $z=\frac{x-x_o}{x_o}$, we get to the equation \begin{equation}\label{e3}
		f(z)=|z+1|^{\beta}-z\beta-1=\frac{\lambda}{x_o^\beta\cos\theta},
	\end{equation}
	and we note that, for $\theta\in[0,\tilde\theta)$, both $x_o$ and $\cos\theta$ are non-negative. By applying Taylor's formula with integral remainder to $f$, we obtain
	\begin{equation*}
	f(z)=\beta(\beta-1)\int_{0}^{z}(1+t)^{\beta-2}(z-t)\de t.
	\end{equation*}
	Notice that for $z\in[0,1)$ it holds
	\begin{equation*}
	\int_{0}^{z}(1+t)^{\beta-2}(z-t)\de t\approx \int_{0}^{z}(z-t)\de t\approx z^2.
	\end{equation*}
	On the other hand, for $z\in[1,+\infty)$ it holds
	\begin{equation*}
	\begin{split}
		\int_{0}^{z}(1+t)^{\beta-2}(z-t)\de t&=\int_{0}^{z/2}(1+t)^{\beta-2}(z-t)\de t+\int_{\frac{z}{2}}^{z}(1+t)^{\beta-2}(z-t)\de t\\
		&\approx z\int_{0}^{z/2}(1+t)^{\beta-2}\de t+z^{\beta-2}\int_{\frac{z}{2}}^{z}(z-t)\de t\\
		&=\frac{z}{\beta-1}\left(\left(1+z/2\right)^{\beta-1}-1\right)+z^{\beta-2}\frac{z^2}{8}\approx_\beta z^{\beta}.
	\end{split}
	\end{equation*}
	Hence, we get
	\begin{equation*}
	f(z)\approx_\beta\begin{cases}
		z^2&\text{if}\quad0\leq z<1\\
		z^\beta&\text{if}\quad z\geq1
	\end{cases},
	\end{equation*}
	and if we consider \eqref{e3}, by applying Lemma~\ref{r2}, and by the fact that for $\theta\in[0,\tilde\theta)$ it holds $\cos\theta\approx1$, it follows that
	\begin{equation}\label{e4}
	\frac{x_+-x_o}{x_o}\approx_\beta\begin{cases}
		\lambda^{1/2}x_o^{-\beta/2}&\text{if} \quad 0\leq \lambda x_o^{-\beta}<1\\
		\lambda^{1/\beta}x_o^{-1}&\text{if}\quad \lambda x_o^{-\beta}\geq1
	\end{cases}.
	\end{equation}
	Finally, by the definition of $x_o$, we have 
	\begin{equation*}
		\beta x_o^{\beta-1}=\left.\frac{d}{dx}x^\beta\right|_{x=x_o}=\tan\theta,
	\end{equation*}
	and therefore, we get that for $\theta\in[0,\tilde\theta)$ it holds
 \begin{equation*}
     x_o\approx \theta^{\frac{1}{\beta-1}}.
 \end{equation*}
 The conclusion follows by a rearrangement of the terms in \eqref{e4}.
\end{proof}

\section*{Acknowledgements} I am grateful to my advisors, Luca Brandolini, Leonardo Colzani, Giacomo Gigante, and Giancarlo Travaglini, for their support and all the valuable discussions.

\bibliography{main.bib}
\bibliographystyle{alpha}

\end{document}